\documentclass{article}

\usepackage{amsfonts, amsmath,amssymb, amsthm}
\usepackage{verbatim}
\usepackage[all,2cell]{xypic}
\objectmargin+{1mm}
\labelmargin+{0.8mm}
\SelectTips{cm}{12}

\theoremstyle{plain}  
\newtheorem{thm}{Theorem}[section]

\newtheorem{cor}[thm]{Corollary}
\newtheorem{lem}[thm]{Lemma}
\newtheorem{prop}[thm]{Proposition}

\theoremstyle{definition}

\newtheorem{df}[thm]{Definition}
\newtheorem{ex}[thm]{Example}

\newtheorem{nt}[thm]{Notations}

\newtheorem{rem}[thm]{Remark}
\newtheorem{rev}[thm]{Review}

\theoremstyle{remark}

\DeclareMathOperator{\id}{id}

\DeclareMathOperator{\isoto}{\overset{\scriptstyle{\sim}}{\to}}

\DeclareMathOperator{\rdef}{\twoheadrightarrow}

\DeclareMathOperator{\rinc}{\hookrightarrow}
\newcommand{\ssm}{\smallsetminus}

\renewcommand{\Im}{\operatorname{Im}}
\newcommand{\coker}{\operatorname{Coker}}

\newcommand{\Tot}{\operatorname{Tot}}

\DeclareMathOperator{\Codim}{Codim}
\newcommand{\depth}{\operatorname{depth}}
\newcommand{\grade}{\operatorname{grade}}
\newcommand{\hight}{\operatorname{ht}}

\newcommand{\kos}{\operatorname{Kos}}
\DeclareMathOperator{\pd}{Projdim}

\newcommand{\rank}{\operatorname{rank}}
\DeclareMathOperator{\Spec}{Spec}
\DeclareMathOperator{\Supp}{Supp}
\newcommand{\Sym}{\operatorname{Sym}}

\newcommand{\END}{\operatorname{\mathcal{END}}}

\DeclareMathOperator{\HOM}{\mathcal{HOM}}

\newcommand{\Ar}{\operatorname{Ar}}

\newcommand{\dom}{\operatorname{dom}}
\newcommand{\ran}{\operatorname{ran}}
\newcommand{\op}{\operatorname{op}}

\DeclareMathOperator{\i1in}{i_1,\cdots,\underset{j}{1},\cdot,i_n}

\DeclareMathOperator{\dd}{\mathfrak{d}}
\DeclareMathOperator{\ee}{\mathfrak{e}}
\DeclareMathOperator{\ff}{\mathfrak{f}}

\DeclareMathOperator{\ii}{\mathfrak{i}}

\DeclareMathOperator{\mm}{\mathfrak{m}}
\DeclareMathOperator{\pp}{\mathfrak{p}}

\DeclareMathOperator{\vv}{\mathfrak{v}}
\DeclareMathOperator{\xx}{\mathfrak{x}}

\newcommand{\cA}{\mathcal{A}}

\newcommand{\cC}{\mathcal{C}}

\newcommand{\cM}{\mathcal{M}}

\newcommand{\cP}{\mathcal{P}}

\newcommand{\phib}{\phi_{\bullet}}
\newcommand{\xb}{x_{\bullet}}
\newcommand{\yb}{y_{\bullet}}

\newcommand{\AddCat}{\operatorname{\bf AddCat}}
\newcommand{\Set}{\operatorname{\bf Set}}

\newcommand{\Cat}{\operatorname{\bf Cat}}
\newcommand{\Cub}{\operatorname{\bf Cub}}

\newcommand{\GKos}{\operatorname{\bf GKos}}

\DeclareMathOperator{\Homo}{H}

\newcommand{\onto}[1]{\stackrel{#1}{\to}}

\newcommand{\Ch}{\operatorname{\bf Ch}}

\newcommand{\cMA}{\cM_A}
\newcommand{\cMAI}{\cM_A^I}

\newcommand{\Kos}{\operatorname{\bf Kos}}

\newcommand{\unit}{\operatorname{unit}}

\newcommand{\Wt}{\operatorname{\bf Wt}}

\newcommand{\WtAI}{\Wt_A^I}

\newcommand{\Z}{\mathbb{Z}}

\renewcommand{\coprod}{\sqcup}

\def\sn{\smallskip\noindent}

\title{Generalized Koszul resolutions}

\author{Satoshi Mochizuki and Akiyoshi Sannai}
\date{}

\begin{document}
\maketitle
\begin{abstract}
The main objective of this paper is 
to generalize a notion of Koszul resolutions and 
charcterizing modules which admits such a resolution. 
We turn out that for a noetherian ring $A$ and a coherent $A$ module $M$, 
$M$ has a two dimensional generalized Koszul resolution 
if and only if $M$ is a pure weight two module in the sense of \cite{HM09}. 
\end{abstract}

\section{Introduction}

The main point of this paper is to 
deal with new theory of complexes and resolutions. 
More precisely, 
the theory of generalized Koszul resolutions 
and related this notion with pure weight modules 
defined in \cite{HM09} and Koszul cubes defined in \cite{Kos}. 
To state the main theorem now we define the notion 
of generalized Koszul resolution and 
Koszul cubes. 

\begin{df}[\bf Generalized Koszul resolutions]
A $n$-cube $x$ in a category $\cC$ is a contravariant functor 
from ${[1]}^{\times n}$ to $\cC$ where $[1]$ is a totally ordered set 
$\{0,1\}$ with the natural order $0 <1$. 
For each $\ii \in {[1]}^{\times n}$, 
we call $x(\ii)$ a {\it vertex of $x$}. 
For each $\ii=(i_1,\cdots,i_n)$, 
we write $x(\ii-\ee_k \to \ii)$ by $d_{\ii}^k$ 
where $\ee_k$ is the $k$-th unit vector and 
we assume that $i_k=1$. 

\sn
Let $f_1,\cdots,f_n$ be a regular sequence in $A$.\\ 
A {\it Koszul cube} associated with $f_1,\cdots,f_n$ is  
a $n$-cube $x$ in the category of finite $A$-modules satisfying the following conditions:

\sn
(1) each vertex of $x$ is a finite free $A$-modules and 
their rank is constant. 
Therefore we can consider the determinant of its boundary maps. 

\sn 
(2) there are positive integers $m_1,\cdots,m_n$ and 
$\det d_{\ii}^k=f_k^{m_s}$ for each 
$\ii=(i_1,\cdots,i_n)\in{[1]}^{\times n}$ 
and $1\leqq k \leqq n$ such that $i_k=1$. 

\sn
A {\it generalized Koszul resolution} associated with 
$f_1,\cdots,f_n$ is the totalized complex of a Koszul cubes 
associated with $f_1,\cdots,f_n$. 
\end{df}

Now we state the main theorems:

\begin{thm}[\bf Theorem~\ref{thm:main 1}]
For any finite $A$-module $M$ and a regular sequence $f$, $g$, 
the following conditions are equivalent:

\sn
{\rm (1)} $M$ is pure weight two modules supported on $V(f,g)$ 
in the sense of {\rm \cite{HM09}}. 
Namely, $M$ is supported on $V(f,g)$ and of projective dimension less than two. 

\sn
{\rm (2)} $M$ is resolved by a generalized Koszul resolution 
associated with $f$, $g$. 

\sn
{\rm (3)} There is a Koszul cube $x$ associated with $f$, $g$ 
such that $\Homo_0(\Tot x)$ is isomorphic to $M$.
\end{thm}

\sn
{\it Conventions.} Throughout of this paper, 
we use the letter $A$ 
to denote a commutative noethrian ring with a unit. 
We denote the category of finite $A$-modules by $\cM_A$. 

\medskip\sn
{\it Acknowledgements.} 
The second author is grateful 
for Kei-ichi Watanabe and Takafumi Shibuta 
for stimulating discussion. 

\section{Definition of weight}

In this section, 
we start from reviewing a notion 
of pure weight perfect modules 
over noetherian rings. 
For more information 
of pure weight perfect modules 
over any schemes, 
see \cite{HM09}. 
Mainly we intend to study 
fundamental properties 
of pure weight modules 
over a Cohen-Macaulay local ring. 

\begin{df}
For an ideal $I$ of $A$ generated 
by a regular sequence of $A$ 
and $0\leqq r \leqq \infty$, 
let us denote the category 
of finitely generated $A$-modules 
of projective dimension $\leqq r$ 
supported on $V(I)$ 
by $\cMAI(r)$. 
One can easily check 
that $\cMAI(r)$ 
is closed under the extensions 
and direct summand in $\cMA$. 
Therefore $\cMAI(r)$ has 
a natural exact categorical structure. 
If $I$ is generated by a regular sequence 
consisting of $r$ elements, 
we denote $\cMAI(r)$ by $\WtAI$ 
and an $A$-module in $\WtAI$ 
is said to be {\it pure weight} $r$ $A$-module 
({\it supported on $I$}).
\end{df}

\begin{rem}
The definition above is compatible with 
that in \cite{HM09}. 
To prove this we shall notice 
that the notion of torsion and projective dimension 
are equivalent 
for a finitely generated modules 
on a noetherian ring 
(see \cite[Proposition 4.1.5]{Wei94}) 
and that in the notation above, 
$\Spec A/I \rinc \Spec A$
is a regular closed immersion. 
\end{rem}

\begin{ex}

\sn
Let $f_1,\cdots,f_p$ be a regular sequence of $A$, 
then $A/(f_1,\ldots,f_p)$ is a typical example 
of a module of weight $p$. 
We will call such a pure weight module  
a {\it{simple pure weight module}}.

\sn
A module of weight $0$ is a projective module 
whose support is total space $\Spec A$.

\sn
If $A$ is a Cohen-Macaulay local ring of Krull dimension $d$, 
a module of weight $d$ is just a module 
of finite projective dimension and finite length.

\sn
If $A$ is not a Cohen-Macaulay ring, 
the class of $A$-modules of finite length 
and finite projective dimension 
does not work fine as in 
the following example in \cite{Ger74}. 
Let $(A,\mm)$ be a $2$-dimensional local ring 
which is not normal and $\Spec A \ssm \{\mm\}$ is regular. 
Then an $A$-Module of finite length 
and finite projective dimension is the zero $A$-module.

\begin{proof}[\bf Rough sketch of proof]
Let $M$ be a non-zero, $A$-module of finite length 
and finite projective dimension. 
Since $A$ is not Cohen-Macaulay, 
we have inequalities  
$$\pd_A M \leqq \pd_A M +\depth_A M =\depth A < 2.$$
Let 
$$A^{\oplus n} \onto{\phi} A^{\oplus m} \to M \to 0$$
be a projective (= free) resolution of $M$. 
Then we can easily notice 
that $n=m$ and 
let us put 
$f=\det \phi$ and 
let $\pp$ be one of minimal prime of the ideal $(f)$. 
Then by Krull's theorem, 
we have $\hight \pp=1$ and 
we can easily notice that $M_{\pp} \neq 0$. 
But by hypothesis we have $\Supp M=\{\mm\}$. 
This is contradiction.
\end{proof}

In this section, 
from now on, 
we assume that 
$A$ is a Cohen-Macaulay ring. 

\begin{nt}
Let us denote the category of weight $p$ $A$-modules 
by $\Wt^p_A$. 
Since $\Wt^p_A$ is closed under extensions in $\cM_A$, 
$\Wt^p_A$ can be naturally considered as an exact category. 
\end{nt}

In this section from now on, we moreover assume that 
$A$ is local ring 
of Krull dimension $d$

\begin{prop}
Let $M$ be a non-zero, 
pure weight $p$ $A$-module. 
Then $M$ is a Cohen-Macaulay module of dimension $d-p$. 
\end{prop}

\begin{proof}[\bf Proof]
We have two equalities:
\begin{equation}
\dim_A M + \Codim_A M= d
\end{equation}
\begin{equation}
\pd_A M + \depth_A M= d
\end{equation}
Therefore we have 
$$d-p \leqq \depth_A M \leqq \dim_A M \leqq d-p.$$
Hence we get $\depth_A M = \dim_A M =d-p$. 
\end{proof}

\end{ex}

\begin{cor}
For a non-zero pure weight module $M$, 
its associated prime ideal is minimal.
\end{cor}

\begin{proof}[\bf Proof]
It is a general property of Cohen-Macaulay modules. 
\end{proof}

\section{Generality of cubes}

In this section, we fix a general notion of cubes. 
From now on, let $\cC$ be a category. 

\begin{nt}
For a set $S$, we denote its power set by $\cP(S)$. 
That is, $\cP(S)$ is the set of all subset of $S$. 
By the natural inclusion order, 
we can consider $\cP(S)$ as an order sets, a fortiori, a category. 
For a map $f:S \to T$ between sets, 
we define its {\it push forward} which is an order preserving map as follows.
$$f_{\ast}:\cP(S)\ni x \to f(x) \in \cP(T).$$
Obviously the association above is functorial, 
we get the functor 
$$\cP:\Set \to \Cat$$
where $\Set$, $\Cat$ are the category of sets, 
the category of categories respectively.
\end{nt}

\begin{ex}
\label{ex:correspondences}
For a positive integer $n$, 
we denote the set of positive integers $k$ 
such that $1\leqq k \leqq n$ by
$(n]$ and for a non-negative integer $m$, 
we denote the totally order set of integers $k$ 
such that $0\leqq k \leqq n$ with the natural order by $[n]$. 
We put $\cP_n=\cP((n])$. 
Then we have the canonical category isomorphism 
$$\cP_n \isoto [1]^{\times n}$$
$$S \mapsto (\chi_S(1),\cdots,\chi_S(n))$$
where $\chi_S$ is the {\it characteristic function} 
associated with $S$, 
that is, the function $\chi_S:(n] \to [1]$ defined by 
$$ \chi_S(k)=
\begin{cases}
1 & \text{if $k \in S$}\\
0 & \text{otherwise}
\end{cases}
.$$
\end{ex}

\begin{ex}
\label{ex:coprodisom}
For a pair of disjoint sets $S$, $T$, 
we have the canonical category equivalence:
$$\cP(S)\times\cP(T)\ni(A,B) \mapsto A\cup B \in \cP(S\coprod T).$$
\end{ex}

\begin{df}
For a set $S$, 
a {\it $S$-cube} in $\cC$ is a contravariant functor from 
$\cP(S)$ to $\cC$. 
We denote the category of $S$-cubes in $\cC$ by
$$\Cub^S(\cC):=\HOM({\cP(S)}^{\op},\cC).$$
Since the association $\HOM({\cP(S)}^{\op},-)$ is covariant, 
we get the functors 
$$\Cub^S:\Cat \to \Cat$$
$$\Cub:\Set \ni S \mapsto \Cub^S \in \END(\Cat):=\HOM(\Cat,\Cat).$$
\end{df}

\begin{nt}
Let $S$ be a set and $x$ be a $S$-cube in $\cC$. 
For $T \in \cP(S)$ and $k\in T$, 
we denote $x(S)$ by $x_S$ and call 
it a {\it vertex of $x$}. 
we also write $x(T\ssm\{k\} \rinc T)$ by 
$d_T^{x,k}$ or shortly $d_T^k$ and call 
it a {\it boundary morphism of $x$}.
\end{nt}

The following lemma is sometimes useful to treat morphisms of cubes.

\begin{lem} 
\label{lem:genofcubemap} 
Let $S$ be a set.

\sn
{\rm (1)} 
For any $S$-cube $x$, 
every $T$, $U\in\cP(S)$ such that 
$T \subset U$ and $U \ssm T$ is a finite set, 
$x(U \subset T)$ is described as composition of boundary morphisms.

\sn
{\rm (2)} 
Assume that $S$ is a finite set. 
For any $S$ cubes $x$, $y$ and a family of morphisms 
$f={\{f_T:x_T \to y_T\}}_{T\in\cP(S)}$ in $\cC$, 
$f:x \to y$ is a morphism of $S$-cubes in $\cC$ 
if and only if for any $T\in\cP(S)$ and $k\in T$, 
we have the equality $d_T^{y,k}f_T=f_{T\ssm\{k\}}d_T^{x,k}$. 
\end{lem}

\begin{ex} 
\label{ex:lowcardinalcube} 
Let $S$ be a set. 

\sn
{\rm (1)} 
If $S$ is the empty set $\emptyset$, 
then we have the canonical isomorphism of endfunctors on $\Cat$:
$$\Cub^{\emptyset}\isoto \id_{\Cat}.$$

\sn
{\rm (2)} 
We denote the morphism category of $\cC$ is 
the category whose objects are morphisms in $\cC$ 
and morphisms are commutative squares in $\cC$ 
by $\Ar\cC$. 
The assignment $\cC \mapsto \Ar\cC$ above is functorial, 
therefore we have the endofunctor: 
$$\Ar:\Cat \to \Cat.$$ 
We have the canonical two natural transformations 
$$\dom,\ \ran:\Ar \to \id_{\Cat}.$$
Namely, for any morphism $f:x\to y$ in $\cC$, 
we have $\dom(f)=x$ and $\ran(f)=y$. 

\sn
If $S$ is a singleton $\{s\}$, 
then we have the canonical isomorphism of endfunctors on $\Cat$:
$$\Cub^{\{s\}}\isoto\Ar.$$ 
\end{ex}

\begin{nt} 
For a positive integer $n$, 
$(n]$-cubes are simply said to be 
$n$-cubes and $\Cub^{(n]}$ is denoted by 
$\Cub^n$. 
By the isomorphism in \ref{ex:correspondences}, 
we sometimes consider $n$-cubes as 
contravariant functors on $[1]^{\times n}$. 
Then for a $n$-cube $x$ in $\cC$, 
we denote $x(i_1,\cdots,i_n)$ by $x_{i_1,\cdots,i_n}$ and 
$x((i_1,\cdots,\overset{j}{\hat{1}},\cdots,i_p) 
\to (i_1,\cdots,\overset{j}{\hat{0}},\cdots,i_p))$ 
by $d^{x,j}_{\i1in}$ or shortly $d^j_{\i1in}$.
\end{nt}

\begin{ex}
\label{ex:associativity}
For a pair of disjoint sets $S$, $T$, 
a $S$-cube in $T$-cube is canonically considered as a $S\coprod T$-cube. 
More precisely, utilizing the isomorphism in \ref{ex:coprodisom}, 
we have the canonical isomorphism between endfunctors on $\Cat$
$$\Cub^{S\coprod T}\isoto\Cub^S(\Cub^T).$$ 
\end{ex}

\begin{ex}
\label{ex:typicalkoscube}
Let $f_1,\cdots,f_n$ be elements in $A$. 
We define the $n$-cube $\kos(f_1,\cdots,f_n)$ 
(or shortly $\kos(\ff_1^n)$) in $\cP_A$ 
by 
$\kos(\ff_1^n)_S:=A \ \ \text{for any $S \in \cP_n$}$ and 
$d_S^j=f_j \ \ \text{for any $S \in \cP_n$ and $j \in S$}.$
\end{ex}

\begin{nt}
\label{nt:faceofcube}
For a set $S$ and an element $k \in S$, 
we have the two kind of canonical natural transformations 
$\delta_{S,\dom}^k$, $\delta_{S,\ran}^k:\Cub^S \to \Cub^{S\ssm\{k\}}$ defined by the composition
$$\Cub^S(\cC)\isoto\Cub^{\{k\}}(\Cub^{S\ssm\{k\}}(\cC))
\isoto \Ar(\Cub^{S\ssm\{k\}}(\cC))
\to \Cub^{S\ssm\{k\}}(\cC)$$
where the first morphism is defined in \ref{ex:associativity}, 
the second one is described in \ref{ex:lowcardinalcube} 
and the last one is $\dom$ or $\ran$ respectively. 
For a $S$-cube $x$ in $\cC$, 
$\delta_{S,\dom}^k(x)$, 
$\delta_{S,\ran}^k(x)$ 
are said to be the {\it domain side $k$-face}, 
{\it range side $k$-face} of $x$ respectively. 
\end{nt}

The following lemma is sometimes useful for treating cubes.

\begin{lem}[\bf Cube lemma]
\label{lem:cube lemma}
For the diagram in the category $\cC$:
$$\xymatrix{
a \ar[rrr] \ar[ddd] 
& & & b \ar[ddd]\\ 
& x \ar[r] \ar[d] \ar[ul] 
& y \ar[d] \ar[ur] 
& \\
& z \ar[dl] \ar[r] 
& w \ar[dr] \\
c \ar[rrr] & & &
d & , 
}$$
assume that the morphism $\vec{wd}$ is a monomorphism 
(resp. $\vec{xa}$ is an epimorphism) 
and every squares except $xywz$
(resp. $abdc$) 
are commutative. 
Then $xywz$ 
(resp. $abdc$) 
is also commutative.
\end{lem}

From now on, 
we assume that $\cC$ has a zero object.

\begin{df}[\bf Homology of cubes]
\label{df:homologycube}
Let $S$ be a set, $x$ a $S$-cube in $\cC$ 
and an element $k$ in $S$. \\
Assume that for any $T \in \cP(S\ssm\{k\})$, 
there exist a cokernel (resp. a kernel) of $d_{T\cup \{k\}}^k$, 
then the {\it $k$-direction $0$-th} 
(resp. {\it $1$-th}) {\it homology} of $x$ 
is a $S\ssm\{k\}$-cube in $\cC$ 
denoted by $\Homo^k_0(x)$ (resp. $\Homo^k_1(x)$) and defined by 
$\Homo_0^k(x)_T:=\coker d_{T\cup \{k\}}^k$ 
(resp. $\Homo_1^k(x)_T:=\ker d_{T\cup \{k\}}^k$).\\
There are the natural quotient and inclusion morphisms 
$\Homo_1^k(x) \rinc x_{T\cup\{k\}}$ and 
$x_T \rdef \Homo_0^k(x)$. 
The associations above are functorial, 
that is, 
if we assume that 
any morphism in $\cC$ admits its cokernel (resp. kernel), 
then we have the following functor 
$$\Homo_0^k \text{(resp. $\Homo_1^k$)}:\Cub^S(\cC) \to 
\Cub^{S\ssm\{k\}}(\cC)$$ 
and the natural transformations 
$$\Homo_1^k \onto{\iota_S^k} \delta_{S,\dom}^k \text{   and   }  
\delta_{S,\ran}^k \onto{\pi_S^k} \Homo_0^k.$$
\end{df}

\section{Notations for multi complexes}

Now we are fixing the notation of multi complexes in an additive category.
We assume that $\cA$ is an additive category. 

\begin{nt}
\label{nt:mapstoZ}
For a set $S$, 
let us denote the set of maps from $S$ to $\Z$ by $\Z^S$. 
It has the natural term-wise addition, that is, 
for any $f$, $g$ in $\Z^S$, we define $f+g$ by $(f+g)(s):=f(s)+g(s)$ 
for any $s$ in $S$. 
Let us recall that for any $T \in \cP(S)$, 
$\chi_T \in \Z^S$ is the characteristic function 
associated with $T$ 
(see \ref{ex:correspondences}). 
In particular for any $s\in S$, 
$\chi_{\{s\}}$ is denoted by $\delta^s$. 
\end{nt}

\begin{df}[\bf Multi complexes]
\label{df:multicomplex} 
For a set $S$, 
a {\it $S$-multi complexes} $\xb$ in $\cA$ is a family 
$\{x_f,d^{x,s}_f\}_{f\in\Z^S,s\in S}$ consisting of 
objects $x_f$ in $\cC$ 
and morphisms $d_f^{x,s}:x_f \to x_{f-\delta^s}$ 
satisfying the following equalities: 
$$
\begin{cases}
d^{x,s}_{f-\delta^s}d^{x,s}_f=0\\
d^{x,t}_{f-\delta^s}d^{x,s}_f+d^{x,s}_{f-\delta^t}d^{x,t}_f=0
\end{cases}
$$
for every $f\in\Z^S$ and $s\neq t$ in $S$.\\ 
For each $f\in \Z^S$ and $s\in S$, 
$d^{x,s}_f$ is said to be {\it boundary map} and 
sometimes written by 
$d^s_f$ shortly. 

\sn
A $S$-multi complex $\xb$ is {\it bounded} if 
except finitely many $f \in \Z^S$, 
$x_f=0$. 

\sn
A morphism of $S$-multi complexes $\phib:\xb \to \yb$ is a family of morphisms 
$\{\phi_f:x_f \to y_f\}_{f\in\Z^S}$ satisfying the equality: 
$$d_f^{y,s}\phi_f=\phi_{f-\delta^s}d_f^{x,s}$$
for any $f\in\Z^S$ and $s$ in $S$.
Let us denote the category of $S$-multi complexes on $\cA$ by $\Ch^S(\cA)$ and 
the category of bounded $S$-multi complexes on $\cA$ by $\Ch^S_b(\cA)$.\\ 
Therefore we get the functors 
$$\Ch^S,\ \Ch^S_b:\AddCat \to \AddCat$$ 
where $\AddCat$ is the category of additive categories. 
\end{df}

\begin{ex}
For the empty set $\emptyset$, 
$\Ch^{\phi}(\cA)=\cA$ and 
for an one element set $\{s\}$, 
we write $\Ch^{\{s\}}$ by $\Ch$. 
$\Ch(\cA)$ is the usual category of chain complexes in $\cA$ 
by identifying $\Z^{\{s\}}=\Z$. 
In this case, 
for a complex in $\Ch(\cA)$, 
we simply denote its boundary maps $d_n^{x,s}$ by $d_n^x$ or shortly $d_n$. 
\end{ex}

\begin{rem}[\bf Sign notations]
\label{rem:sign notation}
For a non-negative integer $n$, 
we have the canonical isomorphism of functors
$$\Ch^{(n]} \isoto \Ch^n$$
where $\Ch^n$ is the $n$-times iteration of the functor $\Ch$.
For the $(n]$-multi complex $\xb$ in $\cA$, 
we put ${d'}_f^k:={(-1)}^{\overset{n}{\underset{t=k+1}{\sum}}f(t)}d_f^k$ 
for each $f\in\Z^{(n]}$ and $k\in(n]$. 
Then one can easily verify $\{x_f,{d'}^k_f\}$ is an object in $\Ch^n(\cA)$. 
This correspondence gives the isomorphism above. 
By the identification above, 
we say the $(n]$-multi complexes as 
the $n$-multi complexes.\\ 
the $n$-th cubes in $\cA$ is 
also naturally considered as 
the $n$-th multi complexes in $\cA$. 
\end{rem}

\begin{ex}
\label{ex:associativitymulticomp}
For a pair of disjoint finite sets $S$, $T$, 
by utilizing the isomorphisms in \ref{rem:sign notation}, 
we have a non-canonical isomorphism of functors 
$$\Ch^{S\coprod T}\isoto\Ch^S(\Ch^T),$$
$$\Ch_b^{S\coprod T}\isoto\Ch_b^S(\Ch_b^T).$$
\end{ex}

\begin{df}[\bf Total functors] 
For a finite set $S$, 
we define the total functor 
$$\Tot^S:\Ch_b^S \to \Ch_b$$
as follows. 
$\Tot^S(x)_n:=\underset{\underset{k\in S}{\sum} f(k)=n}{\bigoplus} x_f$
and $d_n:=\underset{\substack{\ \ \ \ \ \ \ s\in S \\  
\underset{k\in S}{\sum}} 
f(k) =n}{\sum} d_f^{s}$.\\
For a pair of disjoint finite sets $S$ and $T$, 
we fix isomorphisms in \ref{ex:associativitymulticomp}, 
$$\alpha:\Ch_b^{S\coprod T}\isoto\Ch_b^S(\Ch_b^T),$$
$$\beta:\Ch_b(\Ch_b^T)\isoto\Ch^{T\coprod\{\ast\}}$$
where $\ast$ is a symbol which is not in $T$. 
Then we have the functor
$\Tot^S_{\alpha,\beta}:\Ch^{S\coprod T} \to \Ch^{T\coprod \{\ast\}}$ 
by the composition of functors below:
$$\Ch_b^{S\coprod T} \onto{\alpha} \Ch_b^S(\Ch_b^T) \onto{\Tot^S} 
\Ch_b(\Ch_b^T) \onto{\beta} \Ch^{T\coprod\{\ast\}}.$$  
We often omit $\alpha$ and $\beta$ in the notation.
\end{df}

\section{Admissible cubes}

In this section, we will study a specific class of cubes, 
so called, {\it admissible cubes} on an abelian category.
From now on, let us fix an abelian category $\cA$ and $S$ a finite set. 
We start from considering a typical example.

\begin{df}
\label{df:A-seq}
A regular sequence $f_1,\cdots,f_n$ is {\it $A$-sequence} 
if $f_{\sigma(1)},\cdots,f_{\sigma(n)}$ is also regular sequence 
for any bijection $\sigma$ on $(n]$.
\end{df}

\begin{rem}
\label{rem:Aseqcriterion}
For any regular sequence $f_1,\cdots,f_n$, 
if the condition below is satisfied, 
then $f_1,\cdots,f_n$ is an $A$-sequence.

\sn
For each $i$, 
$A/(f_1,\cdots,f_i)A$ is complete for $(f_1,\cdots,f_n)$-adic topology.
\end{rem}

\begin{ex}
\label{ex:admcubeex}
For a sequence $f_1$, $\cdots$, $f_n$ in $A$, 
it is $A$-sequence if and only if for any $k$ in $[n-1]$ 
and distinct numbers 
$i_1$, $\cdots$, $i_k$ in $(n]$, 
a map 
$f_{i_{k+1}}:A/(f_{i_1},\cdots,f_{i_k}) \to A/(f_{i_1},\cdots,f_{i_k})$ is injective. 
This is equivalent to the $n$-cube $\kos(\ff_1^n)$ 
(for definition, see \ref{ex:typicalkoscube})
satisfies the following condition:\\ 
For any $k$ in $[n-1]$ and distinct numbers 
$i_1$, $\cdots$, $i_k$ in $(n]$, 
boundary maps of $\Homo_0^{i_1}(\cdots (\Homo_0^{i_k}(\kos(\ff_1^n)))\cdots)$ are injections.\\
In the case, $\kos(\ff_1^n)$ is said to be 
the {\it typical Koszul cubes} associated with $f_1,\cdots,f_n$. 
\end{ex}

\begin{df}[\bf Admissible cubes]
\label{df:admcubes}
Let $x$ be a $S$-cube. 
If $\sharp S=1$, 
$x$ is called {\it admissible} if its boundary morphism is a monomorphism. 
Inductively, for $\sharp S>1$, 
$x$ is called {\it admissible} if its boundary morphisms are monomorphism 
and if for every $k$ in $S$, $\Homo^k_{0}(x)$ is admissible. 
\end{df}

\begin{lem}
\label{lem:coincided} 
For an admissible $S$-cube $x$ in $\cA$ 
and distinct elements $i_1,\cdots,i_k$ in $S$, 
we have the canonical isomorphism:
$$\Homo_0^{i_1}(\Homo_0^{i_2}(\cdots(\Homo_0^{i_k}(x))\cdots)) \isoto 
\Homo_0^{i_\sigma(1)}(\Homo_0^{i_\sigma(2)}(\cdots(\Homo_0^{i_\sigma(k)}(x))\cdots))$$ 
where $\sigma$ is a bijection on $S$. 
\end{lem}

\begin{proof}[\bf Proof]
It is enough to prove the case for $S=\{1,2\}$. 
Then by $3 \times 3$-lemma, we turn out that $\Homo^2(\Homo^1(x))$ and 
$\Homo^1(\Homo^2(x))$ is canonical isomorphic to the object $y$ in the diagram below. 
$$\xymatrix{
x_{11} \ar@{>->}[r] \ar@{>->}[d] & 
x_{01} \ar@{->>}[r] \ar@{>->}[d] & 
{\Homo^2_0(x)}_1 \ar@{>->}[d]\\
x_{01} \ar@{>->}[r] \ar@{->>}[d] & 
x_{00} \ar@{->>}[r] \ar@{->>}[d] & 
{\Homo^2_0(x)}_0 \ar@{->>}[d]\\
{\Homo^1_0(x)}_1 \ar@{>->}[r] &
{\Homo^1_0(x)}_0 \ar@{->>}[r] &
y.}$$
\end{proof}

\begin{nt}
For an admissible $S$-cube $x$ in $\cA$ and a 
subset $T=\{i_1,\cdots,i_k\} \subset S$, 
we put 
$\Homo^T_0(x):=\Homo_0^{i_1}(\Homo_0^{i_2}(\cdots(\Homo_0^{i_k}(x))\cdots))$.
\end{nt}

\begin{prop}
\label{prop:tothomology}
Let $x$ be a $S$-cube in $\cA$. 
Assume that every $k$-face of $x$ is admissible, 
then for any $k$ in $S$, 
we have 
$$\Homo_p(\Tot^{S}(x))\isoto 
\begin{cases}
\Homo_p(\Homo_0^{S\ssm\{k\}}(x)) & \text{for $p=0,\ 1$}\\
0 & \text{otherwise}
\end{cases}
.$$
In particular if we assume that 
$x$ is admissible, 
then we have
$$\Homo_p(\Tot^{S}(x))=   
\begin{cases}
\Homo^{S}_0(x) & \text{for $p=0$}\\
0 & \text{otherwise}
\end{cases}
.$$
\end{prop}

\begin{proof}[\bf Proof]
We intend to use induction for the number of elements in $S$. 
For $\sharp S=1$, the assertion is trival. 
Now let us consider for $n\sharp S>1$. 
For simplicity, we put $T=S\ssm\{k\}$. 
Let us consider the spectral sequence associated with 
the bi-complex $\Tot^T(x)$, 
$$\Homo_p^v\Homo_q^h(\Tot^T(x)) 
\Rightarrow \Homo_{p+q}(\Tot^S(x)).$$
Considering the short exact sequences 
associated with the spectral sequence above 
$$0 \to \Homo_0^v\Homo_q^h(\Tot^T(x)) \to \Homo_q(\Tot^S(x)) 
\to \Homo_1^v\Homo_{q-1}^h(\Tot^T(x)) \to 0,$$
and noticing that by inductive hypothesis, 
we have the isomorphisms
$$\Homo_q^h(\Tot^T(x))\isoto
\begin{cases}
\Homo_0^T(x) & \text{for $q=0$}\\
0 & \text{otherwise}
\end{cases},$$
we get the isomorphisms
$$\Homo_p(\Tot^S(x))\isoto 
\begin{cases}
\Homo_p^v(\Homo_0^T(x)) & \text{for $p=0,\ 1$}\\
0 & \text{otherwise}
\end{cases}
.$$
Hence we obtain the assertion.
\end{proof}

\section{Generalized Koszul resolution}

In this section, we introduce a generalization of 
Koszul resolution. 
Let us start from reviewing {\it{Koszul complexes}}.

\begin{rev}[\bf Koszul complex]
Let $f:P \to Q$ be a $A$-module homomorphism 
between projective $A$-modules. 
The {\it $n$-th Koszul complex} 
associated with $f$ is denoted by $\kos^n(f)$ 
and defined as follows:
$${\kos^n(f)}_k=\Lambda^k(P)\otimes\Sym^{n-k}(Q)$$
and the Koszul differential 
$d_{k+1}:{\kos^n(f)}_{k+1} \to {\kos^n(f)}_k$ 
is given by 
$$p_1\wedge\cdots\wedge p_{k+1} \otimes q_{k+2}\cdots q_n \mapsto
\overset{k+1}{\underset{i=1}{\sum}}
(-1)^{k+1-i}
p_1\wedge\cdots\wedge\hat{p_i} \wedge\cdots \wedge p_{k+1} \otimes 
f(p_i) q_{k+2}\cdots q_n
.$$
\end{rev}

\begin{ex}
Let $f_1,\cdots,f_n$ be elements in $A$. 
We consider a homomorphism 
$\ff=
\begin{pmatrix}
f_1 & \cdots & f_n
\end{pmatrix}
:A^{\oplus n} \to A$. 
For simplicity, 
the $n$-th Koszul complex associated with $\ff$ 
is denoted by $\kos(\ff)$. 
Notice that $\Tot^{(n]}(\kos(\ff_1^n))$, 
the total complex of $\kos(\ff_1^n)$ defined in \ref{ex:typicalkoscube}, 
is isomorphic to $\kos(\ff)$ above.\\ 
It is well-known that if $f_1,\cdots,f_n$ forms regular sequence, 
then $\Homo_k(\kos(\ff))$ is trivial for $k \geqq 1$. 
In this section, we intend to generalize this fact.
\end{ex}

\begin{df}[\bf Generalized Koszul complexes]
\label{df:GKC}
Let $n$, $m$ be positive integers and $f_1,\cdots,f_n$ elements in $A$. 
For a family of endomorphisms on 
$A^{\oplus m}$ $\dd=\{d_S^j\}_{S\in\cP_n,\ j\in S}$ such that 
$\det d_S^j=f_j$ for any $S\in\cP_n$ and $j\in S$, 
we define the {\it generalized Koszul complex} associated with $\dd$, 
$\kos(\dd)$ as follows.\\ 
$\kos(\dd)_k:=\underset{\substack{S\in\cP_n \\ 
\sharp S=k}}{\bigoplus} F_S$ where $F_S:=A^{\oplus m}$ and 
its boundary maps are defined by 
$$(-1)^{\overset{n}{\underset{t=j+1}{\sum}}\chi_S(t)}d_S^j:F_S \to F_{S\ssm\{j\}}$$on its $F_S$ component. 
\end{df}

\begin{rem}
The definition in \ref{df:GKC} is equivalent to 
that $\kos(\dd)$ is the total complex of the following $n$-cube $x$:\\ 
$x_S:=A^{\oplus m}$ and $d_S^s:=d_S^s$ for any $S\in\cP_n$ and $s\in S$.\\ 
Therefore we also denote $x$ by $\kos(\dd)$. 
\end{rem}

Let us recall that 
a complex $E_{\bullet}$ on an abelian category 
is said to be {\it $n$-spherical} 
if $\Homo_k(E_{\bullet})=0$ unless $k\neq n$. 

\begin{thm}
\label{thm:resolcriterion}
For a family $\dd$ as in the notation \ref{df:GKC} 
if $f_1,\cdots,f_n$ forms a regular sequence, 
$\kos(\dd)$ is $0$-spherical. 
\end{thm}

To prove the theorem above, we need to use Buchsbaum and Eisenbud Theorem. 
To state the theorem, we start from recalling the notation of ideal of minors. 

\begin{df}[\bf Ideal of minors]
\label{df:idealminor}
Let $U$ be an $m \times n$ matrix over $A$ 
where $m$, $n$ are positive integers. 
For $t$ in $(\min(m,n)]$ we then denote by $I_t(U)$ the ideal generated by 
the $t$-minors of $U$, that is, 
the determinant of $t\times t$ sub-matrices of $U$.\\ 
For an $A$-module homomorphism 
$\phi:M\to N$ between finite free $A$-modules, 
let us choose a matrix representation $U$ 
with respect to bases of $M$ and $N$. 
One can easily prove that the ideal $I_t(U)$ only depend on $\phi$. 
Therefore we put $I_t(\phi):=I_t(U)$. 
\end{df}

\begin{thm}[\bf Buchsbaum-Eisenbud \cite{BE73}]
\label{thm:BEthm}
For a complex of finite free $A$-modules
$$F_{\bullet}:0 \to F_s \onto{\phi_s} F_{s-1} \onto{\phi_{s-1}} \to \cdots \to 
F_1 \onto{\phi_1} F_0 \to 0,$$ 
set $r_i=\overset{s}{\underset{j=i}{\sum}}(-1)^{j-i} \rank F_j$. 
Then the following are equivalent:

\sn
{\rm (1)} $F_{\bullet}$ is $0$-spherical.

\sn
{\rm (2)} $\grade I_{r_i}(\phi_i) \geqq i$ for any $i$ in $(s]$. 
\end{thm}

\begin{proof}[\bf Proof of Theorem~\ref{thm:resolcriterion}]
We denote boundary maps 
of $\kos(\dd)$ and $\kos(\ff_1^n)$ 
by $d_k^{\dd}$ and $d_k^{\ff}$ respectively. 
We also put 
$$r_i=\overset{n}{\underset{j=i}{\sum}}(-1)^{j-i} \rank \kos(\ff_1^n)_j 
=\overset{n}{\underset{j=i}{\sum}}(-1)^{j-i} 
\begin{pmatrix}
n\\
j
\end{pmatrix}.$$
Then we have 
$$\overset{n}{\underset{j=i}{\sum}}(-1)^{j-i} \rank \kos(\dd_1^n)_j=mr_i.$$
In the notation \ref{df:GKC} if $m=1$, 
then $\kos(\dd)=\kos(\ff)$ and 
in this case, 
the assertion is well-known. 
Therefore by \ref{thm:BEthm}, 
it follows that $\grade I_{r_i}(d_i^{\ff}) \geqq i$ for any $i$ in $(n]$. 
But inspection shows that for each $i\in(n]$, 
$I_{r_i}(d_i^{\ff})\subset I_{nr_i}(d_i^{\dd})$. 
Therefore we use Theorem~\ref{thm:BEthm} again, 
we turn out that $\kos(\dd)$ is $0$-spherical. 
\end{proof}

\begin{cor}
\label{cor:admcriterion}
For a family $\dd$ as in the notation \ref{df:GKC}, 
if $f_1,\cdots,f_n$ forms $A$-sequences, 
then the cube $\kos(\dd)$ is admissible.
\end{cor}

\begin{proof}[\bf Proof]
We prove by induction for $n$.
In the case for $n=1$, the assertion is trivial. 
For $n>1$, 
by inductive hypothesis, 
we notice that all $k$-faces of $\kos(\dd)$ is admissible for $k\in(n]$. 
For each $k$, 
then by \ref{prop:tothomology}, 
we have the isomorphism:
$$\Homo_1(\Homo_0^{(n]\ssm\{k\}}(\kos(\dd)))\isoto
\Homo_1(\Tot^{(n]}(\kos(\dd)))$$ 
and by \ref{thm:resolcriterion}, 
we turn out that the groups above is trivial. 
It means that $\kos(\dd)$ is admissible.
\end{proof}

\begin{df}[\bf Generalized Koszul resolutions]
\label{df:GKR}
Let $f_1,\cdots,f_n$ be an $A$-sequence.\\ 
A {\it Koszul} ($n$-){\it cube} associated with $f_1,\cdots,f_n$ is  
a $n$-cube $x$ in $\cM_A$ satisfying the following conditions:

\sn
(1) each vertex of $x$ is a finite free $A$-modules and 
their rank is constant. 
Therefore we can consider the determinant of its boundary maps. 

\sn 
(2) there are positive integers $m_1,\cdots,m_n$ and 
$\det d_S^s=f_s^{m_s}$ for each $S\in\cP_n$ and $s\in S$. 

\sn
A {\it generalized Koszul resolution} associated with 
$f_1,\cdots,f_n$ is the totalized complex of a Koszul cubes 
associated with $f_1,\cdots,f_n$. 
Let us denote the category of Koszul cubes 
(resp. generalized Koszul resolution) 
associated with 
$f_1,\cdots,f_n$ by $\Kos_A^{\ff_1^n}$ (resp. $\GKos_A^{\ff_1^n}$). 
\end{df}

\begin{nt}[\bf Rank and determinants of Koszul cubes]
\label{nt:rankanddeterminat}
For an $n$-Koszul cube $x$, 
we define {\it rank of $x$} by $\rank x:=\rank x_{\emptyset}$. 
We also define {\it $j$-th determinant of $x$} by $\det_j x:=\det d^j_{\{j\}}$ for any $j\in (n]$.   
\end{nt}

\begin{rem}
By \ref{thm:resolcriterion}, 
a generalized Koszul resolution is $0$-spherical 
and by \ref{cor:admcriterion}, 
Koszul cube is an admissible cubes.\\ 
As in the notation \ref{df:GKR}, 
we have the totalized functor
$$\Tot:\Kos_A^{\ff_1^n} \to \GKos_A^{\ff_1^n}$$
which is essentially surjective and faithful. 
\end{rem}

\begin{rem}[\bf Compatibility of the definition in \cite{Kos}]
\label{rem:comdef}
Now let us assume that 
$A$ is a noetherian ring such that 
every finite projective $A$-module is free. 
For example, 
the $n$-th valuable polynomial ring over a principle ideal domain or 
local ring and so on.\\ 
Then in \cite{Kos}, 
the notion of Koszul cubes is defined as follows:

\sn
A Koszul $n$-cube $x$ associated with an $A$-sequence $f_1,\cdots,f_n$  
is an admissible $n$-cube in $\cM_A$ which satisfies the following conditions.

\sn
(1) each $X_S$ is free $A$-modules for any $S\in\cP_n$.

\sn
(2) For each $i\in(n]$, the vertexes of $\Homo_0^k(x)$ are in $\Wt_A^{(f_i)}$. 

\sn
Utilizing the lemma~\ref{lem:equivalenceboudarycond} below, 
we turn out that the definition~\ref{df:GKR} 
and the definition above is equivalent 
if we assume that all $f_i$ is a prime element. 
\end{rem}

\begin{lem}
\label{lem:equivalenceboudarycond}
Let $f$ be a non-zero divisor. 
For an $A$-module homomorphism 
$$\psi:A^{\oplus n} \to A^{\oplus n}$$
Let us consider the following assertions.

\sn
{\rm (1)} $\det\psi=f^{\alpha} \times \unit$ for some $\alpha$.

\sn
{\rm (2)} $\coker \psi$ is in $\Wt_A^{(f)}$ and $\psi$ is injective.

\sn
The assertion {\rm (1)} implies the assertion {\rm (2)}. 
Moreover if we assume that $f$ is a prime element, 
then the assertion {\rm (2)} implies the assertion {\rm (1)}.
\end{lem}

\begin{proof}[\bf Proof]
First we assume that the assertion {\rm (1)}. 
There is the cofactor morphism $\tilde{\psi}$ of $\psi$ such that 
$\tilde{\psi}\psi$ is multiplication by $f^{\alpha} \times \unit$. 
Hence $\psi$ is injective and 
$\coker\psi$ has projective dimension one. 
Since $\psi_f$ is an isomorphism, 
we turn out that $\coker\psi$ is supported on $V(f)$. 
Now we assume the assertion {\rm (2)} and that $f$ is a prime element, 
then we notice that $\psi_f$ is an isomorphism and 
it means that $\det \psi$ is in $A^{\times}_f$. 
Now we turn out that $\det\psi=f^{\alpha} \times \unit$ for some $\alpha$.
\end{proof}

\section{Weight two cases}

In this section, 
we assume that $A$ is noetherian commutative ring with unit 
and let $f$, $g$ be an $A$-sequence. 
First we give a characterization of weight two modules. 

\begin{thm}
\label{thm:main 1}
For any $M$ in $\cM_A$, 
the following conditions are equivalent.

\sn
{\rm (1)} $M$ is in $\Wt^{(f,g)}_A$. 

\sn
{\rm (2)} $M$ is resolved by a generalized Koszul resolution 
associated with $f$, $g$. 

\sn
{\rm (3)} There is a Koszul cube $x$ in $\Kos^{f,g}_A$ such that 
$\Homo_0(\Tot x)$ is isomorphic to $M$.
\end{thm}

Obviously the assertion that (2) implies (1) and that (3) implies (2) 
are trivial. 

\begin{proof}[\bf Proof] 
Let us fix an $A$-module $M$ in $\Wt_A^{(f,g)}$.
By replacing $f$ with $f^{\alpha}$ for some $\alpha$ and so on, 
without loss of generality, 
we may assume that $fM=gM=0$. 
Then now 
we have a surjection $(A/f)^{\oplus n} \to M$ with kernel $L$. 
Since $\Wt_A^{(f)}$ is closed under taking kernels of surjections, 
$L$ is in $\Wt_A^{(f)}$. 
By considering resolutions of $L$ and $(A/f)^{\oplus n}$, 
we get the following diagram. 
$$\xymatrix{
A^{\oplus m} \ar[r] \ar[d]_{P} & A^{\oplus n} \ar[d]^f\\ 
A^{\oplus m} \ar[r]_{U} & A^{\oplus n}
}$$ 

\sn
\textbf{Claim} $\det P=$ unit$\times f^n$. 

\begin{proof}[\bf Proof of \textbf{Claim}]
First Localize at $f$ for the sequence below 
$$A^{\oplus m} \overset{P}{\to} A^{\oplus m} \to L \to 0,$$ 
we notice that $\frac{\det p}{f^n}$ is in $A_f^{\times}$. 
Next localize at $g$ for the sequence above again, 
we get the following diagram:
$$\xymatrix{
A^{\oplus n}_{g} \ar[r]^f \ar[d] & A^{\oplus n}_{g} \ar[r] \ar[d] & 
(A_{g}/f)^{\oplus n} \ar[d]^{\wr}\\
A^{\oplus m}_{g} \ar[r]_{P_{g}} & A^{\oplus m}_{g} \ar[r] & L_{g}. 
}$$
Now taking any prime ideal $\pp$ in $\Spec A_g$ 
and localize the diagram above at $\pp$. 
Then since the top line is a minimal resolution, 
the vertical morphism as 
a complex is a split quasi-isomorphism. 
Therefore we turn out that 
$\frac{\det P}{f^n}$ is in ${(A_g)}_{\pp}^{\times}$ 
for any prime ideal $\pp$ in $A$ and hence 
we notice that $\frac{\det P}{f^n}$ is in $A^{\times}_g$. 
Since $f$, $g$ forms $A$-sequence, 
we have the equality $A^{\times}_f \cap A^{\times}_g =A^{\times}$ 
in the quotient ring of $A$. 
Therefore we obtain the result.
\end{proof}

Now to get a Koszul cube, we are arranging the square above. 

\sn
\textbf{Claim 2} There are $n\times m$ matrix $X$ and $n\times n$ matrix $V$ 
such that $UX=gE_n+fV$ where $E_n$ is the $n$-th unit matrix. 

\begin{proof}[\bf Proof of \textbf{Claim 2}]
For each $k\in (n]$, 
let us denote one of a pull back 
of $g\ee_k$ in ${(A/f)}^{\oplus n}$ by the maps 
$$A^{\oplus m} \rdef L  \rinc {(A/f)}^{\oplus n}$$
by $\xx_k= 
\begin{pmatrix}
x_{k1}\\
\vdots\\
x_{km}
\end{pmatrix}
 \in A^{\oplus m}$. 
Then there is a vector 
$\vv_k=
\begin{pmatrix}
v_{k1}\\
\vdots\\
v_{kn}
\end{pmatrix}
\in A^{\oplus n}$ such that 
$U\xx_k=g\ee_k+f\vv_k$ 
where $\ee_k$ is the $k$-th unit vector in $A^{\oplus n}$. 
We put $X=
\begin{pmatrix}
\xx_1 &\cdots & \xx_n
\end{pmatrix}$ 
and 
$V=
\begin{pmatrix}
\vv_1 &\cdots & \vv_n
\end{pmatrix}
$.
\end{proof}

Put the matrix $\bar{U}$ as follows:
$$\bar{U}=
\begin{pmatrix}
fV & U \\
X & E_m
\end{pmatrix}
$$
where $E_m$ is the $m$-th elementary matrix. 
Since we have 
$$
\begin{pmatrix}
-gE_n & 0\\
0 & E_m
\end{pmatrix}
=
\begin{pmatrix}
E_n & -U\\
0 & E_m
\end{pmatrix}
\bar{U}
\begin{pmatrix}
E_n & 0\\
-X & E_m
\end{pmatrix}
$$
it follows that $\det \bar{U}={(-g)}^n$. 
Now the following diagram is desired $x$: 
$$\xymatrix{
A^{\oplus m+n} \ar@{-->}[r] \ar[d]_{
\begin{pmatrix}
E_n & 0\\
0 & P
\end{pmatrix}
} 
& A^{\oplus m+n} \ar[d]^{
\begin{pmatrix}
fE_n & 0\\
0 & E_m
\end{pmatrix}
}\\ 
A^{\oplus m+n} \ar[r]_{\bar{U}} & A^{\oplus m+n}
}$$
where 
dotted map is induced from the commutative diagram below:
$$\xymatrix{
A^{\oplus m+n} \ar[r]^{\bar{U}} \ar@{->>}[d] & A^{\oplus m+n} \ar@{->>}[d]\\ 
L \ar[r] & {(A/f)}^{\oplus n} .
}$$
\end{proof}

\begin{cor}
\label{cor:strwt2}
For any $M$ in $\Wt^{(f,g)}_A$, 
there are endomorphisms $P$, $Q:A^{\oplus n} \to A^{\oplus n}$ 
such that $M$ is isomorphic to $\frac{A^{\oplus n}}{<\Im P,\Im Q>}$ 
and $P$ and $Q$ are similar to the following matrixes:
$$P \sim
\begin{pmatrix}
fE_m & 0\\
0 & E_{n-m}
\end{pmatrix} 
$$
$$Q \sim
\begin{pmatrix}
gE_m & 0\\
0 & E_{n-m}
\end{pmatrix} 
$$
where $f$, $g$ forms a regular sequence. 
\end{cor}

\end{document}